\documentclass[submission,copyright,creativecommons]{eptcs} 
\providecommand{\event}{ACT 2023}

\usepackage{underscore}
\usepackage[T1]{fontenc}

\title{Posetal Diagrams for Logically-Structured\\Semistrict Higher Categories}
\newcommand{\titlerunning}{Posetal Diagrams for Semistrict Higher Categories}
\newcommand{\authorrunning}{C. Sarti, J. Vicary}

\hypersetup{
  bookmarksnumbered,
  pdftitle    = {\titlerunning},
  pdfauthor   = {\authorrunning},
  pdfsubject  = {Posetal Diagrams},
  pdfkeywords = {Posetal Diagrams, Higher Categories, Semistrictness}
}

\usepackage{amsmath}
\usepackage{amsthm}
\usepackage{amssymb}
\usepackage{mathrsfs}
\usepackage{tikz-cd}
\usepackage{tikz}
\usetikzlibrary{calc} 
\usepackage{xspace}
\usepackage[font=it]{caption}
\usepackage[numbers,sort]{natbib}

\newcommand{\Set}{\ensuremath{\mathrm{Set}}}
\newcommand{\Cat}{\ensuremath{\mathrm{Cat}}}

\newcommand{\FinPos}{\ensuremath{\mathrm{FPos}}}
\newcommand{\FinDLat}{\ensuremath{\mathrm{FDLat}}}
\newcommand{\Func}{\ensuremath{\mathrm{Func}}}
\newcommand{\Op}{\ensuremath{\mathrm{op}}}

\newcommand{\cat}[1]{\ensuremath{\mathrm{#1}}}
\newcommand{\func}[1]{\ensuremath{#1}}
\newcommand\constr[1]{\mathsf{#1}}
\newcommand{\id}{\ensuremath{\mathrm{id}}}
\newcommand{\ev}{\ensuremath{\mathrm{ev}}}

\newtheorem{nthm}{Theorem}[section]
\newtheorem{nlemma}[nthm]{Lemma}
\newtheorem{nprop}[nthm]{Proposition}
\newtheorem{ncor}[nthm]{Corollary}
\theoremstyle{definition}
\newtheorem{ndef}[nthm]{Definition}
\newtheorem{nexample}[nthm]{Example}
\newtheorem{nconstr}[nthm]{Construction}

\newcommand{\homotopyio}{\textsc{homotopy.io}\xspace}

\tikzset{curve/.style={settings={#1},to path={(\tikztostart)
    .. controls ($(\tikztostart)!\pv{pos}!(\tikztotarget)!\pv{height}!270:(\tikztotarget)$)
    and ($(\tikztostart)!1-\pv{pos}!(\tikztotarget)!\pv{height}!270:(\tikztotarget)$)
    .. (\tikztotarget)\tikztonodes}},
    settings/.code={\tikzset{quiver/.cd,#1}
        \def\pv##1{\pgfkeysvalueof{/tikz/quiver/##1}}},
    quiver/.cd,pos/.initial=0.35,height/.initial=0}

\tikzset{tail reversed/.code={\pgfsetarrowsstart{tikzcd to}}}
\tikzset{2tail/.code={\pgfsetarrowsstart{Implies[reversed]}}}
\tikzset{2tail reversed/.code={\pgfsetarrowsstart{Implies}}}
\tikzset{no body/.style={/tikz/dash pattern=on 0 off 1mm}}

\tikzset{reg/.style={dashed}}
\tikzset{sing/.style={}}

\newcommand{\figintrolinearzigzaglevels}{
\begin{tikzpicture}[xscale=1, thick, yscale=1.2, scale=.8]
\begin{pgfscope}
\draw [reg] (-.5,-1) to +(1.5,0);
\draw [sing] (-.5,0) to +(1.5,0);
\draw [reg] (-.5,1.) to +(1.5,0);
\draw [sing] (-.5,2) to +(1.5,0);
\draw [reg] (-.5,3) to +(1.5,0);
\node [thick, fill=white, circle, draw, inner sep=-5pt, minimum width=12pt] at (.25,0) {$a$};
\node [thick, fill=white, circle, draw, inner sep=-5pt, minimum width=12pt] at (.25,2) {$b$};
\end{pgfscope}
\node [left] at (-.5,-1) {$R_{\{\}}$};
\node [left] at (-.5,0) {$S_{a}$};
\node [left] at (-.5,1) {$R_{\{a\}}$};
\node [left] at (-.5,2) {$S_{b}$};
\node [left] at (-.5,3) {$R_{\{a,b\}}$};
\end{tikzpicture}
}

\newcommand{\figintrolinearzigzagzintervals}{
\begin{tikzpicture}[scale=1.25, font=\normalsize, scale=.8]
\node (0) at (-2,0) {$R_{\{\}}$};
\node (1) at (-1,1) {$S_{a}$};
\node (2) at (-2,2) {$R_{\{a\}}$};
\node (5) at (-1,3) {$S_{b}$};
\node (8) at (-2,4) {$R_{\{a,b\}}$};
\draw [->] (1) to (0);
\draw [->] (1) to (2);
\draw [->] (5) to (2);
\draw [->] (5) to (8);
\begin{scope}[black!40]
\node (4) at (0,2) {$S_{a,b}$};
\draw [->] (4) to (1);
\draw [->] (4) to (5);
\node [rotate=-135] at (-.7,2) {$\lrcorner$};
\end{scope}
\end{tikzpicture}
}

\newcommand{\figintroposetalzigzaglevels}{
\begin{tikzpicture}[xscale=2, thick, yscale=1.2]
\begin{pgfscope}
\path [clip, draw=none] (-.5,-2.05) rectangle +(2,4.1);
\draw [sing] (-1,-1) to +(3,3);
\draw [sing] (-1,-2) to +(3,3);
\draw [reg] (-1,-1.5) to +(3,3);
\draw [sing] (-1,2) to +(3,-3);
\draw [sing] (-1,1) to +(3,-3);
\draw [reg] (-1,1.5) to +(3,-3);
\draw [sing] (-1,0) to +(3,0);
\draw [reg] (-1,-2) to +(3,0);
\draw [reg] (-1,0) to +(3,0);
\draw [reg] (-1,2) to +(3,0);
\node [thick, fill=white, circle, draw, inner sep=-5pt, minimum width=12pt] at (0,0) {$x$};
\node [thick, fill=white, circle, draw, inner sep=-5pt, minimum width=12pt] at (1,0) {$y$};
\end{pgfscope}
\node [left] at (-.5,-2) {$[\{\},\{\}]$};
\node [left] at (-.5,-1.5) {$[\{\},\{y\}]$};
\node [left] at (-.5,-1) {$[\{y\},\{y\}]$};
\node [left] at (-.5,-.5) {$[\{y\},\{{x,}y\}]$};
\node [left] at (-.5,0) {$[\{\}, \{x,y\}]$};
\node [left] at (-.5,.5) {$[\{\},\{x\}]$};
\node [left] at (-.5,1) {$[\{x\},\{{x}\}]$};
\node [left] at (-.5,1.5) {$[\{x\},\{x,y\}]$};
\node [left] at (-.5,2) {$[\{x,y\},\{x,y\}]$};
\end{tikzpicture}
}
\newcommand{\figintroposetalzigzagintervals}{
\begin{tikzpicture}[scale=1.25,font=\small]
\node (0) at (0,0) {$[\{\},\{\}]$};
\node (1) at (-1,1) {$[\{\},\{x\}]$};
\node (2) at (-2,2) {$[\{x\},\{x\}]$};
\node (3) at (1,1) {$[\{\},\{y\}]$};
\node (5) at (-1,3) {$[\{x\},\{x,y\}]$};
\node (6) at (2,2) {$[\{y\},\{y\}]$};
\node (7) at (1,3) {$[\{y\},\{x,y\}]$};
\node (8) at (0,4) {$[\{x,y\},\{x,y\}]$};
\draw [->] (1) to (0);
\draw [->] (3) to (0);
\draw [->] (1) to (2);
\draw [->] (3) to (6);
\draw [->] (5) to (2);
\draw [->] (7) to (6);
\draw [->] (5) to (8);
\draw [->] (7) to (8);
\begin{scope}[black!40]
\node (4) at (0,2) {$[\{\},\{x,y\}]$};
\draw [->] (4) to (1);
\draw [->] (4) to (3);
\draw [->] (4) to (5);
\draw [->] (4) to (7);
\node [rotate=45] at (.75,2) {$\lrcorner$};
\node [rotate=-135] at (-.75,2) {$\lrcorner$};
\node [rotate=-45] at (0,1.25) {$\lrcorner$};
\node [rotate=135] at (0,2.75) {$\lrcorner$};
\end{scope}
\end{tikzpicture}
}
\newcommand{\figintervaldiag}{
\begin{tikzcd}[ampersand replacement=\&, row sep=12pt]
\&\& \textcolor{black!40}{[a,e]} \\
\& \textcolor{black!40}{[a,d]} \&\& \textcolor{black!40}{[b,e]} \& \textcolor{black!40}{[c,e]} \\
{[a,b]} \& {[a,c]} \& {[b,d]} \& {[c,d]} \& {[d,e]} \\
{[a,a]} \& {[b,b]} \& {[c,c]} \& {[d,d]} \& {[e,e]}
\arrow[from=3-3, to=4-4]
\arrow[from=3-3, to=4-2]
\arrow[from=3-1, to=4-2]
\arrow[from=3-1, to=4-1]
\arrow[from=3-2, to=4-1]
\arrow[from=3-2, to=4-3]
\arrow[from=3-4, to=4-3]
\arrow[from=3-4, to=4-4]
\arrow[color=black!40, from=2-2, to=3-3]
\arrow[color=black!40, from=2-2, to=3-1]
\arrow[color=black!40, from=2-2, to=3-2]
\arrow[color=black!40, from=2-2, to=3-4]
\arrow[from=3-5, to=4-4]
\arrow[from=3-5, to=4-5]
\arrow[color=black!40, from=2-4, to=3-3]
\arrow[color=black!40, from=2-4, to=3-5]
\arrow[color=black!40, from=2-5, to=3-4]
\arrow[color=black!40, from=2-5, to=3-5]
\arrow[color=black!40, from=1-3, to=2-2]
\arrow[color=black!40, from=1-3, to=2-4]
\arrow[color=black!40, from=1-3, to=2-5]
\end{tikzcd}
}
\newcommand{\figintervaldiagstringrepr}{
\begin{tikzpicture}[yscale=0.8]
\draw (0, 0.5) -- (0,1);
\draw (0, 5.5) -- (0,4);
\draw (0, 4) .. controls (1.6, 3.0) and (0.667, 1.833) .. (0,1);
\draw (0,1) .. controls (-0.667, 1.833) and (-1.6, 3.0) .. (0,4);
\node [thick, fill=white, circle, draw, inner sep=-5pt, minimum width=12pt] at (0.55,1.75) {$f$};
\node [thick, fill=white, circle, draw, inner sep=-5pt, minimum width=12pt] at (-0.55,1.75) {$g$};
\node [thick, fill=white, circle, draw, inner sep=-5pt, minimum width=12pt] at (0.8,3) {$g$};
\node [thick, fill=white, circle, draw, inner sep=-5pt, minimum width=12pt] at (-0.8,3) {$f$};
\node [thick, fill=white, circle, draw, inner sep=-5pt, minimum width=12pt] at (0,4.7) {$h$};
\end{tikzpicture}
}
\newcommand{\figlimalgorithm}{
\[\begin{tikzcd}[ampersand replacement=\&, row sep=13pt]
{(L,\func{L})} \\
{(L,\func{G}j)} \&\& {\func{F}j} \\
{(L,\func{G}j')} \&\&\& {\func{F}j'} \& {\constr{L}(\cat{C})} \\
\\
L \&\& {(\func{F}j)^0} \&\& \FinPos \\
\&\&\& {(\func{F}j')^0}
\arrow["{\func{F}h}", from=2-3, to=3-4]
\arrow["{(\rho_j,\id)}", from=2-1, to=2-3]
\arrow["{(\rho_{j'},\id)}"', from=3-1, to=3-4]
\arrow["{(\id_L,\func{G}h)}"', from=2-1, to=3-1]
\arrow["{\rho_j}", from=5-1, to=5-3]
\arrow["{\rho_{j'}}"', from=5-1, to=6-4]
\arrow["{(\func{F}f)^0}", from=5-3, to=6-4]
\arrow["{(\id_L,\eta)}"', from=1-1, to=2-1]
\arrow["{\func{U}}"', from=3-5, to=5-5]
\end{tikzcd}\]
}
\newcommand{\figequaliserofpullbacksquares}{
\[\begin{tikzcd}[ampersand replacement=\&, column sep=5pt, row sep=10pt]
\&\& \textcolor{black!40}{\func{Y}[f][d,d']} \&\&\& {\func{Y}[f][a,a']} \\
\& \textcolor{black!40}{\func{X}[d,d']} \&\&\& {\func{X}[a,a']} \\
{\func{W}[d,d']} \&\&\& {\func{W}[a,a']} \\
\&\& {\func{Y}[f][c,c']} \&\&\& {\func{Y}[f][b,b']} \\
\& {\func{X}[c,c']} \&\&\& {\func{X}[b,b']} \\
{\func{W}[c,c']} \&\&\& {\func{W}[b,b']}
\arrow[from=2-5, to=5-5]
\arrow[from=5-2, to=5-5]
\arrow[color=black!40, from=2-2, to=5-2]
\arrow[color=black!40, from=2-2, to=2-5]
\arrow[shift right=1, from=5-2, to=4-3]
\arrow[shift left=1, from=5-2, to=4-3]
\arrow["\Bigg{\lrcorner}"{anchor=center, pos=0.05}, color=black!40, draw=none, from=2-2, to=5-5]
\arrow[shift right=1, from=2-5, to=1-6]
\arrow[shift right=1, from=5-5, to=4-6]
\arrow[color=black!40, shift right=1, from=2-2, to=1-3]
\arrow[color=black!40, shift left=1, from=2-2, to=1-3]
\arrow[shift left=1, from=2-5, to=1-6]
\arrow[shift left=1, from=5-5, to=4-6]
\arrow[from=1-6, to=4-6]
\arrow[color=black!40, from=1-3, to=1-6]
\arrow[color=black!40, from=1-3, to=4-3]
\arrow[from=4-3, to=4-6]
\arrow[from=3-1, to=6-1]
\arrow[from=3-1, to=2-2]
\arrow[from=6-1, to=5-2]
\arrow[from=6-1, to=6-4]
\arrow[from=3-4, to=6-4]
\arrow[from=3-1, to=3-4]
\arrow[from=3-4, to=2-5]
\arrow[from=6-4, to=5-5]
\arrow["\Bigg{\lrcorner}"{anchor=center, pos=0.05}, draw=none, color=black!40, from=1-3, to=4-6]
\end{tikzcd}\]
}

\begin{document}

\author{Chiara Sarti
\institute{Computer Laboratory\\ University of Cambridge}
\email{cs2197@cam.ac.uk}
\and
Jamie Vicary
\institute{Computer Laboratory\\ University of Cambridge}
\email{jamie.vicary@cl.cam.ac.uk}
}

\maketitle

\begin{abstract}
We now have a wide range of proof assistants available for compositional reasoning in monoidal or higher categories which are free on some generating signature. However, none of these allow us to represent categorical operations such as products, equalizers, and similar logical techniques. Here we show how the foundational mathematical formalism of one such proof assistant can be generalized, replacing the conventional notion of string diagram as a geometrical entity living inside an $n$-cube with a posetal variant that allows exotic branching structure. We show that these generalized diagrams have richer behaviour with respect to categorical limits, and give an algorithm for computing limits in this setting, with a view towards future application in proof assistants.
\end{abstract}

\section{Introduction}

The development of proof assistants for category theory and higher category theory has recently been a active area for the applied category theory community, in particular from a string diagrammatic perspective. Recent work has included the \textsc{Cartographer} tool of Sobocinski, Wilson and Zanasi applying hypergraph rewriting for symmetric monoidal diagrams~\cite{cartographer}; the \textsc{DisCoPy} python library from de Felice et al for string diagram manipulation~\cite{discopy}; the \textsc{Rewalt} tool by Hadzihasanovic and Kessler for rewriting with diagrammatic sets~\cite{rewalt}; the \textsc{wiggle.py} tool due to Burton for 3d string diagram rendering~\cite{wigglepy}; the \textsc{Quantomatic} system developed by Dixon, Duncan, Kissinger and others for applying the ZX calculus in quantum information~\cite{quantomatic}; and the \textsc{Globular}~\cite{barglobular2018, barvicary2017} and \textsc{homotopy.io}~\cite{reutterhighlevel2019, heidemann2022} web-based systems for finitely-generated semistrict $n$-categories rendered as higher string diagrams. While these tools represent a wide range of perspectives and use-cases, they share a common goal of allowing the user to manipulate terms in a monoidal or higher categorical structure which is freely generated under composition from some signature, perhaps with additional algebraic elements (for example, such as Frobenius algebras, in the case of \textsc{Quantomatic}.)

The geometrical essence of these proof assistants allows the user to avoid some of the bureaucracy associated with some algebraic approaches to higher categories. However, much of the power of category theory arises from methods that go beyond direct composition, such as products, equalizers, colimits, and other standard categorical structures. These cannot be represented with any of the current family of string diagrammatic proof assistants.

Here we explore an alternative foundation for such proof assistants which may suggest a path towards new classes of tools, with the potential to combine the clarity and usability of string diagrammatic techniques, with the power of algebraic categorical methods. We illustrate our approach with the \textit{zigzag construction}, a simple combinatorial model of higher string diagram that forms the basis of the proof assistant \homotopyio. This construction depends on a contravariant equivalence between the augmented simplex category $\Delta_+$, of finite total orders and monotone functions; and the category $\Delta_=$ of non-empty finite total orders, with monotone functions preserving min and max elements. Elegantly described by Wraith~\cite{wraithusing1993}, the equivalence is representable, acting as $\Delta_+(-,\{ \bot \to \top \} ) : \Delta_+^\Op \to \Delta_=$.

\begin{figure}[t!]
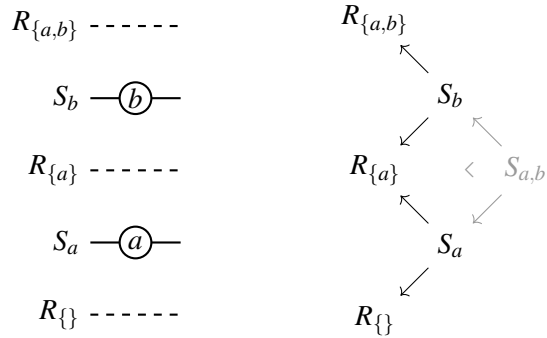

$$
\begin{aligned}
\figintrolinearzigzaglevels
\end{aligned}
\hspace{2cm}
\begin{aligned}
\figintrolinearzigzagzintervals
\end{aligned}
$$
\caption{The linear zigzag over the total order \mbox{$\{a \to b\}$.}}
\label{fig:introlinearzigzag}
\end{figure}
The idea of the zigzag construction is sketched in Figure~\ref{fig:introlinearzigzag}, for the 2-element total order \mbox{$\{a \to b\} \in \Delta_+$}.
On the left, the two elements of $\{a \to b\}$ appear as \textit{singular heights} $S_a$, $S_b$. These are interlaced vertically with the three elements of its dual $\Delta_+(\{a \to b\}, \{\bot \to \top\})$ written as \textit{regular heights} $R_{\{\}}$, $R_{\{a\}}$, $R_{\{a,b\}}$, with the subscript indicating the preimage of $\top$. On the right, we equip the singular heights with maps into the adjacent regular heights.\footnote{These maps are going in the opposite direction than in some previous literature, a technical choice we made following \cite{heidemannframed2022} which makes our development here more straightforward.} We see an alternating pattern of regular and singular heights, drawn as dashed and solid lines; reading from bottom to top, we can consider these as ``time slices'' of the linear geometry of the diagram, which tell us that $a$ happens first, and then $b$. For every category \cat C we have a \textit{zigzag category} $\constr{Z}(\cat C)$, and an $n$\-dimensional string diagram is simply an object of $\constr{Z}^n(\cat C)$ over some suitable base category. For every zigzag category, a functor $R: \constr{Z}(\cat C) \to \Delta_=$ projects down to the total order of regular heights. Concerning the existence of limits in $\constr{Z}(\cat C)$, a decision procedure has recently been presented~\cite{reutterhighlevel2019}, which supplies the following necessary condition: for a diagram in $\constr{Z}(\cat C)$ to have a limit, its projection in $\Delta_=$ must have a limit. This rules out wide classes of limits in $\constr{Z}(\cat C)$, since the limit structure of $\Delta_=$ is meagre. In particular, products do not exist in $\Delta_=$ for sets of cardinality greater than 1, so there can be no notion of product for nontrivial higher string diagrams.

However, the duality $\Delta_+^\Op \simeq \Delta _=$ extends further. In particular, Wraith's original paper further exhibited a duality $\FinPos^\Op \simeq \FinDLat$, where $\FinPos$ is the category of finite posets, and $\FinDLat$ is the category of finite distributive lattices and meet- and join- preserving functors; this can be considered a finitised version of Stone Duality, and is sometimes known as the Birkhoff Representation Theorem~\cite{birkhoffrings1937, wraithusing1993}. This equivalence is again representable as \mbox{$\FinPos(-,\{\bot \to \top\})$}, and extends the duality considered above, given the obvious full and faithful embeddings $\Delta_+ \hookrightarrow \FinPos$ and $\Delta _= \hookrightarrow \FinDLat$.

Here we exploit this duality to put forward a \emph{posetal zigzag construction}, directly generalizing the existing linear zigzag construction. While the linear zigzag construction yields a notion of higher category which is purely compositional, we conjecture that the posetal zigzag construction will yield a notion of higher category with richer categorical structure, but retaining a geometrical essence which could in principle be implemented in a similar tool to \homotopyio. Establishing this conjecture will require considerable future work.

In this paper we take the first steps, giving the first detailed investigation of posetal zigzags. We begin here with an informal illustration of posetal zigzags, generalizing the linear example of Figure~\ref{fig:introlinearzigzag} above. In Figure~\ref{fig:introposetalzigzag} we illustrate the posetal zigzag construction for the poset $\{x\ |\ y\}$, with two disconnected objects. While the linear zigzags in Figure~\ref{fig:introlinearzigzag} had a linear sequence of time slices (since $a$ must precede $b$), here a richer structure appears, with $x$ and $y$ now interpreted as ``events'' that can occur in either order. Now we have a dual distributive lattice $D=\FinPos(\{ x \,\,\,\,y\}, \{\bot \to \top\})$ with elements $\{\}$, $\{x\}$, $\{y\}$, $\{x,y\}$ under the subset order, once again denoting in braces the preimage of $\top$. Our regular and singular levels are indexed by pairs of related elements $[A,B]$ in $D$, which we call \emph{intervals}; by construction, $A,B$ will be subsets of the original poset, with $A \subseteq B$. For this example there are 9 such intervals, and we list them on the left of Figure~\ref{fig:introposetalzigzag}, once again considering them as distinct ``time slices''. For example, $[\{x\}, \{x,y\}]$ represents the time slice in which $x$ has already occurred, and $y$ is in the moment of occurring. More generally, one can interpret any interval $[A,B]$ as the time slice in which the events $A$ has already occurred, and the events $B \setminus A$ are occurring at that exact moment.
\begin{figure}[t!]
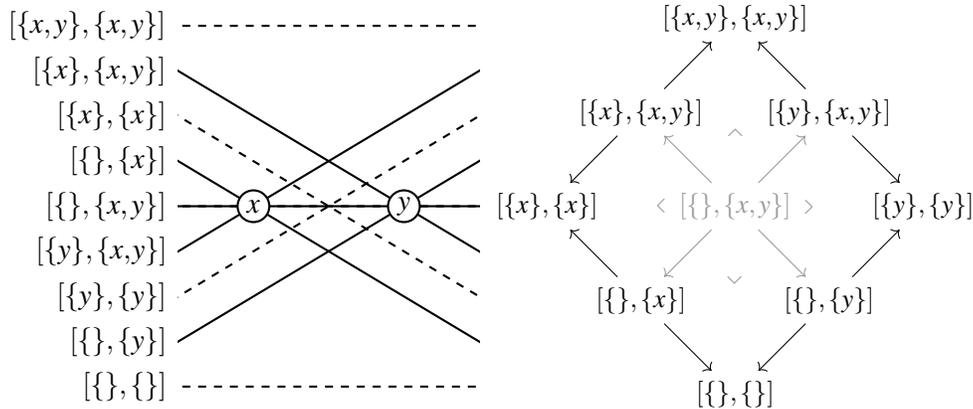

$$
\begin{aligned}
\figintroposetalzigzaglevels
\end{aligned}
\hspace{0.1cm}
\begin{aligned}
\figintroposetalzigzagintervals
\end{aligned}
$$
\caption{The posetal zigzag over the poset $\{ x\ | \ y \}$.}
\label{fig:introposetalzigzag}
\end{figure}

On the right of Figure~\ref{fig:introposetalzigzag}, we present these 9 intervals in a different way, as nodes on a diamond grid, interpreted as a diagram in an underlying process category \cat C. Morphisms arise as interval refinements, and we observe the presence of 4 squares, which are required to commute. As we move from the bottom to the top node by various paths, we observe different sequences of time slices: moving clockwise we observe the sequence $[\{\},\{\}]$, $[\{\},\{x\}]$, $[\{x\},\{x\}]$, $[\{x\},\{x,y\}]$, $[\{x,y\},\{x,y\}]$, interpreted as $x$ occurring before $y$; moving anticlockwise we observe the sequence $[\{\},\{\}]$, $[\{\}, \{y\}]$, $[\{y\},\{y\}]$, $[\{y\},\{x,y\}]$, $[\{x,y\}, \{x,y\}]$, interpreted as $y$ occurring before $x$; and moving vertically one observes the sequence $[\{\},\{\}]$, $[\{\}, \{x, y\}]$, $[\{x,y\}, \{x,y\}]$, in which $x$ and $y$ occur simultaneously. The posetal zigzag diagram provides semantic data in the base category \cat C for all of these possibilities. The fact that the 4 squares commute means that whatever sequence one chooses, the semantic effect in \cat C is the same. The fact that the squares are pullbacks means that much of this data is in fact redundant; the central part, drawn in gray, can be considered as ``filler data'' that need not be explicitly stored. Figure~\ref{fig:introlinearzigzag} contained similar redundant filler data, with the singular level $S_{a,b}$ required to arise as a pullback of $S_a$ and $S_b$.

This account is of course informal, and intended to give intuition ahead of the precise mathematical development that follows. In Section~\ref{sec:intconstr} we introduce the notion of interval for a poset, and give a notion of labelled diagram structure, with label data assigned to all intervals, but with no pullback conditions imposed on filler data. In Section~\ref{sec:posdiag} we introduce the more refined posetal zigzag construction $\constr P (\cat C)$, where the filler data is now obtained canonically via pullback, yielding a well-behaved conservative extension of the traditional linear zigzag formalism. In Section~\ref{sec:limofdiags} we consider the construction of limits in posetal zigzag categories, giving an explicit construction procedure, and establishing our main result, Corollary~\ref{cor:mainresult}, which states that if $\cat C$ has all finite limits, then so does $\constr P(\cat C)$.

In this way, we establish the posetal zigzag construction as a potential foundation for a new class of geometrical proof assistant, with additional expressive power. In this setting, higher string diagrams would retain their semistrict geometrical essence, yet also exhibit new posetal features such as branches, sinks and forks; $n$-dimensional diagrams are no longer inscribed neatly within the $n$-cube, as with traditional higher string diagrams. Depending on the poset structure, the sequencing of morphisms within the string diagram would be dynamic, with re-sequencing steps appearing as higher cells. While the future applications of these ideas of course remain speculative, we hope our work may lead to further development of geometrical proof assistants with exciting new capabilities.

\subsection*{Acknowledgements.}
We thank Lukas Heidemann, Alex Rice and Calin Tataru for many insightful discussions.
We also thank the anonymous reviewers for their invaluable feedback.
The first author acknowledges funding from King's College Cambridge.
The second author acknowledges funding from the Royal Society.

\section{The Interval Construction}
\label{sec:intconstr}

\subsection{From Posets to Labelled Intervals}
\label{subsec:labelledint}

Our formal development begins with a reconstruction of the categorical origins of the zigzag construction, which we generalise to poset shapes.
We will initially focus our analysis on the combinatorial aspects of our theory, setting solid foundations for the establishing the theoretical results of Section~\ref{sec:limofdiags}, and allowing for the geometric content of our theory to emerge organically as a by-product of our categorical analysis.

In our discussion, we study the maps of our diagrams as factored or decomposed in two parts: one which may change the posetal shape but not the labels, and another which may change the labels, but fixes the shape.
By adding the possibility of extra ``filler data'' in our diagrams, this relabelling map can be specified as a natural transformation between two functors.
This sort of decomposition can be expressed naturally in the language of Grothendieck fibrations, unlocking powerful theoretical tools.
Without getting too ahead of our formal development, this Section will build up to a formal description of the construction we depicted in Figure~\ref{fig:introlinearzigzag} and Figure~\ref{fig:introposetalzigzag},
starting from the following key idea:
\begin{ndef}[Interval]
\label{defn:interval}
	An \emph{interval} in a poset $P$ is a pair $[a,a']$ of $a,a' \in P$ with $a \le a'$.
	We denote the set of intervals in $P$ by $[P]$ and order it by the precision relation $\supseteq$
    \begin{align*}
        [a,a'] \supseteq [b,b'] :\longleftrightarrow (a \le b) \land (b' \le a')
	\end{align*}
	The strict counterpart of this relation is denoted $\supset$.
\end{ndef}

\begin{nprop}[]
	For every finite poset $P$, $[P]$ is also a finite poset.
\end{nprop}
\begin{proof}
    Finiteness is evident.
    We must show that $[P]$ inherits transitivity, reflexivity and anti-symmetry from $P$.
    For transitivity, if we have $[a,a'] \supseteq [b,b']$ and $[b,b'] \supseteq [c,c']$,
    then we must have $a \le b \le c$ and $c' \le b' \le a'$,
    so by transitivity of $P$, we have $a \le c$ and $c' \le a'$, and hence $[a,a'] \supseteq [c,c']$.

    Reflexivity is evident, and for anti-symmetry, if we have $[a,a'] \supseteq [b,b']$
    and $[b,b'] \supseteq [a,a']$, then we must have $a \le b \le a$ and $b' \le a' \le b'$, and hence $a = b$ and $a' = b'$ by anti-symmetry of $P$.
\end{proof}

\begin{nexample}[]
\label{eg:diagreprcond}
	The interval construction $[P]$ on the poset $P = \{a< (b\ |\ c) < d < e\}$ is depicted as the left-most diagram below.
    The data of the intervals below the gray arrows is easily rendered as the posetal string diagram on the right.
    This requires us to ignore the witnesses from the higher intervals, which have no clear interpretation on the diagrammatic side.
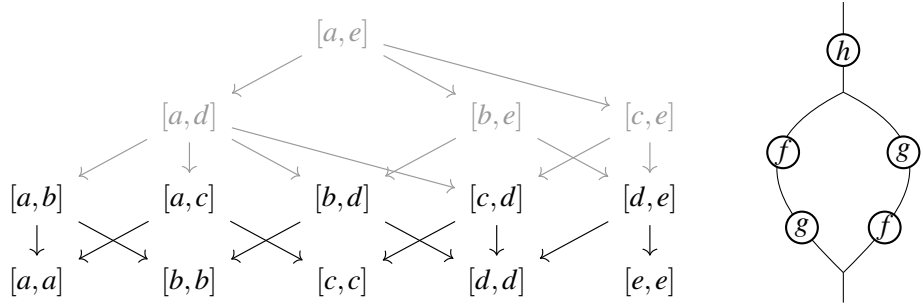
\[\begin{tabular}{c c}
\begin{tabular}{@{}c@{}}
\figintervaldiag
\end{tabular}
&
\begin{tabular}{@{}c@{}}
\figintervaldiagstringrepr
\end{tabular}
\end{tabular}\]
\end{nexample}

With anti-symmetry in mind, we specialise our terminology as follows:
\begin{ndef}[Degenerate Intervals]
	We refer to an interval $[a,a']$ in $P$ as a \emph{degenerate}, if $a=a'$, or \emph{non-degenerate}, if $a < a'$.
\end{ndef}

In the account of this construction offered in \cite{reutterhighlevel2019}, degenerate intervals are called regular heights, and non-degenerate intervals are called singular heights.
Though the alternative choice of terminology is supported by good motivation,
we avoid it in our discussion to spare the reader from unnecessary confusion.

\begin{ndef}[Map of Intervals]
    Let $f: P \to Q$ be a monotone function between posets.
    The associated \emph{map of interval posets} $[f]: [P] \to [Q]$ is the function sending an interval $[a,a']$ in $P$ to $[f(a),f(a')]$.
\end{ndef}

\begin{nprop}[]
	Let $f: P \to Q$ be a monotone function between posets.
    Then $[f]$ is a well-defined monotone map, and moreover, $[-]$ defines an endofunctor on $\FinPos$.
\end{nprop}
\begin{proof}
	Let $f: P \to Q$ be a monotone map.
    Then $a \le a'$ implies $f(a) \le f(a')$, hence if $[a,a']$ is an interval in $P$, $[f][a,a']$ is an interval in $Q$.
	Moreover, if we have $[a,a'] \supseteq [b,b']$, i.e. $a \le b$ and $b' \le a'$,
    then $f(a) \le f(b)$ and $f(b') \le f(a')$, and thus $[f][a,a'] \supseteq [f][b,b']$, which shows $[f]$ is monotone.
    Finally, the assignment $P \mapsto [P]$ and $f \mapsto [f]$ respects identities and composites, and thus makes $[-]$ into an endofunctor $\FinPos$.
\end{proof}

\begin{nlemma}[]
	\label{lemma:prodintintprod}
	The functor $[-]$ preserves products.
\end{nlemma}
\begin{proof}
Let $P$ and $Q$ be finite posets.
An interval in $P \times Q$ is a pair of pairs $[(a,b),(a',b')]$
with $a \le a'$ and $b \le b'$, and thus defines two intervals $[a,a']$ and $[b,b']$.
Moreover, the assignment and its inverse are monotone:
\begin{align*}
	[(a,b),(a',b')] \supseteq [(c,d),(c',d')] &\longleftrightarrow (a,b) \le (c,d) \land (c',d') \le (a',b') \\
    &\longleftrightarrow a \le c \land b \le d \land c' \le a' \land d' \le b' \\
    &\longleftrightarrow [a,a'] \supseteq [c,c'] \land [b,b'] \supseteq [d,d']
.\end{align*}
\end{proof}

We now identify $\FinPos$ with a corresponding full subcategory of $\Cat$,
by identifying each poset $P$ with the category $P$ with the elements of $P$ for objects,
and arrows $a \to a'$, whenever $a \le a'$.
Note that under this identification, the interval construction is isomorphic to the twisted arrow construction \cite[\S2]{johnstonenote1999}.
The above data is combined as follows:
\begin{ndef}[Labelled Interval Category]
	Let $\cat{C}$ be a category. The category $\constr{L}(\cat{C})$ of \emph{intervals labelled in} $\cat{C}$ is defined as the Grothendieck construction of the functor
    \begin{align*}
        \func{L}_{\cat{C}}: \FinPos^\Op &\longrightarrow \Cat \\
        P &\longmapsto \Func([P],\cat{C}), \\
        f &\longmapsto - \circ [f]
    .\end{align*}
\end{ndef}

Explicitly, the objects of $\constr{L}(\cat{C})$ are pairs $(P,\func{X})$,
with $P \in \FinPos$, a shape poset, and $\func{X}: [P] \to \cat{C}$, a labelling of $[P]$ in $\cat{C}$.
As for the morphisms, they are given by pairs $(f,\alpha)$, with $f: P \to Q$, a change of shape map, and $\alpha: \func{X} \to \func{Y} \circ [f]$ a relabelling natural transformation.

\begin{nexample}[]
    \label{eg:labelledintmain}
    Let $P$ be the poset $\{a < (b\,|\,c) < d < e\}$ of Example \ref{eg:diagreprcond} and $\cat{C}$ be the thin category generated by $\{x < (f\,|\,g\,|\,h) < (\alpha\,|\,\beta\,|\,\gamma) < \mu\}$.
    Then the pair $(P,\func{X})$ defines an object of $\constr{L}(\cat{C})$, where $\func{X}: [P] \to \cat{C}$ is the functor represented in the diagram below:
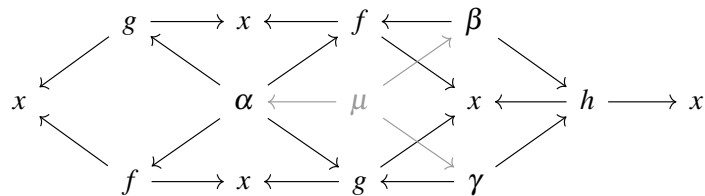
\[\begin{tikzcd}[ampersand replacement=\&, row sep=12pt]
        \& g \& x \& f \& \beta \\
        x \&\& \alpha \& \textcolor{black!40}{\mu} \& x \& h \& x \\
        \& f \& x \& g \& \gamma
        \arrow[from=1-4, to=1-3]
        \arrow[from=1-4, to=2-5]
        \arrow[from=1-2, to=1-3]
        \arrow[from=3-4, to=2-5]
        \arrow[from=3-4, to=3-3]
        \arrow[from=3-2, to=3-3]
        \arrow[from=2-3, to=3-4]
        \arrow[from=2-3, to=1-4]
        \arrow[from=2-3, to=1-2]
        \arrow[from=2-3, to=3-2]
        \arrow[from=2-6, to=2-7]
        \arrow[from=3-5, to=2-6]
        \arrow[from=2-6, to=2-5]
        \arrow[from=3-5, to=3-4]
        \arrow[from=1-5, to=1-4]
        \arrow[from=1-5, to=2-6]
        \arrow[color=black!40, from=2-4, to=3-5]
        \arrow[color=black!40, from=2-4, to=1-5]
        \arrow[color=black!40, from=2-4, to=2-3]
        \arrow[from=3-2, to=2-1]
        \arrow[from=1-2, to=2-1]
\end{tikzcd}\]
\end{nexample}
\begin{nexample}[]
\label{eg:labelledintdisconnected}
    As shapes for cells in a higher category, labelled intervals can admit extremely undesirable behaviour: they need not even be connected.
    For instance, if $P := 1+1$ is the discrete two-element poset,
    then $[P] \cong P$ and thus a labelled interval of shape $P$ simply picks out two objects of $\cat{C}$.
\end{nexample}

\begin{nexample}[]
\label{eg:labelledintmap}
    A map of labelled intervals can be decomposed as a change of shape monotone function and a relabelling natural transformation.
	In the diagram below, the monotone map is between the posets $P := \{\bot \to \top\}$ and $Q := \mbox{\{a < (b\ |\ c) < d\}}$, and acts by $\bot \mapsto a, \top \mapsto d$.
    The relabelling has components $\id_x$, $f \to g$ and $\id_y$.
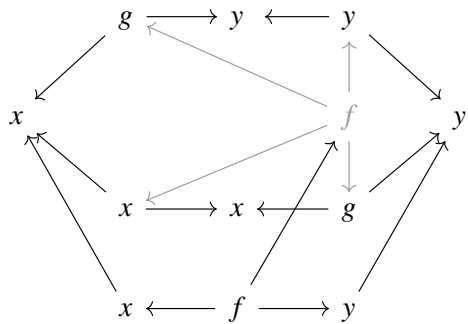
\[\begin{tikzcd}[ampersand replacement=\&]
    \& g \& y \& y \\
    x \&\&\& \textcolor{black!40}{f} \& y \\
    \& x \& x \& g \\
    \& x \& f \& y
    \arrow[from=3-4, to=2-5]
    \arrow[from=3-4, to=3-3]
    \arrow[from=3-2, to=3-3]
    \arrow[from=3-2, to=2-1]
    \arrow[from=1-2, to=2-1]
    \arrow[from=1-2, to=1-3]
    \arrow[from=1-4, to=2-5]
    \arrow[shorten >=3pt, from=1-4, to=1-3]
    \arrow[color=black!40, from=2-4, to=3-4]
    \arrow[color=black!40, from=2-4, to=1-2]
    \arrow[color=black!40, from=2-4, to=3-2]
    \arrow[color=black!40, from=2-4, to=1-4]
    \arrow[from=4-4, to=2-5]
    \arrow[from=4-2, to=2-1]
    \arrow[draw, from=4-3, to=2-4]
    \arrow[from=4-3, to=4-4]
    \arrow[from=4-3, to=4-2]
\end{tikzcd}\]
\end{nexample}

By virtue of its definition through a Grothendieck construction, we have an underlying shape functor $U: \constr{L}(\cat{C}) \to \FinPos$, which is a Grothendieck fibration.
If we adopt the convention of writing $(x^0, x^1)$ for the components of an object $x \in \constr{L}(\cat{C})$, and similarly for morphisms, then the action of $U$ is simply specified by taking first projections $x \mapsto x^0, f \mapsto f^0$.

Given a functor $\func{F}:\cat{C} \to \cat{D}$, we may define a post-composition with $\func{F}$ mapping $\constr{L}(\func{F}): \constr{L}(\cat{C}) \to \constr{L}(\cat{D})$, which acts by $(P,\func{X}) \mapsto (P,\func{F}\circ \func{X})$ and $(f,\alpha) \mapsto (f,\func{F}\alpha)$.
Then $\constr{L}(\func{F})$ is a functor, and moreover it allows us to make an assignment $\constr{L}(-): \cat{C} \mapsto \constr{L}(\cat{C}),\func{F} \mapsto \constr{L}(\func{F})$, which is an endofunctor on $\Cat$, with the usual caveats about size concerns we trust our readers to judge unproblematic.
In particular, this allows us to iterate the construction, as $\constr{L}(\cat{C})$ can itself be taken as a category of labels, and we may thus define iterated versions of this construction by setting $\constr{L}^{n+1}(\cat{C}) := \constr{L}(\constr{L}^n(\cat{C}))$.

It seems useful at this point, before drawing this Section to an end, to briefly summarise our progress towards the overarching goal of posetal diagrams.
The interval construction introduced in Definition \ref{defn:interval} captures the combinatorial aspect of our intended construction remarkably well.
Unfortunately, as we have seen in Example \ref{eg:diagreprcond}, the interval construction introduces labels which are foreign to the diagrammatic calculus.
In keeping with our aim of being able to reconstruct the combinatorial object from the geometric representation, we need to ensure that those labels do not carry information that cannot be inferred from the explicit datum of a diagram.
We will see in the next Section that this canonicity requirement can be succinctly stated in terms of limit constructions.

\section{Posetal Diagrams}
\label{sec:posdiag}

\subsection{Posetal Diagrams as Local Functors}
\label{subsec:posdiaglocal}

In the previous Section we have presented the construction of the category of labelled intervals, arguing that it correctly captures the desiderata combinatorics our theory.
However, this approach requires too much filler data to be specified, preventing a clean diagrammatic presentation of our structures.
In drawing intuition from Example \ref{eg:diagreprcond}, we wish to find technical conditions under which the missing filler data in the diagrams can be faithfully reconstructed.

A close inspection of Example \ref{eg:diagreprcond} and Example \ref{eg:labelledintmain} shows that all the intervals whose data we wish to suppress satisfy the universal property of being a pullback in $[P]$.
This would suggest we consider labellings $\func{X}: [P] \to \cat{C}$ which preserve pullbacks.
However, this condition is far too strong for our interest: $[P]$ is a thin category, and thus every arrow is monic. If $\func{X}$ were to preserve all pullbacks, it could only involve monic arrows in $\cat{C}$, which is exceedingly restrictive, especially in light of our wishes to iterate the construction.
We will thus exercise some care in determining exactly which pullbacks in $[P]$ should be preserved:
\begin{ndef}[Atomic Cospan, Local Functor]
    Let $\cat{J}$, $\cat{C}$ be categories.
    A cospan $(a \stackrel{f}{\to} x \stackrel{g}{\leftarrow} b)$ in $\cat{J}$ is \emph{atomic} if for any cospan $(a \stackrel{f'}{\to} y  \stackrel{g'}{\leftarrow} b)$
    and map $h: y \to x$ with $h \circ f = f'$ and $h \circ g = g'$, then $h$ is an isomorphism.
    We say a functor $\func{X}: \cat{J} \to \cat{C}$ is \emph{local} if it preserves all pullbacks of atomic cospans.
\end{ndef}

\begin{nexample}
    If $f$ is not an isomorphism, then $(a \stackrel{f}{\to} x \stackrel{f}{\leftarrow} a)$ is never atomic.
    Hence non-iso monic arrows need not be preserved by local functors.
\end{nexample}
\begin{nexample}
    The cospan $[b,e] \supset [e,e] \subset [c,e]$ in our Example \ref{eg:diagreprcond} is also not atomic.
\end{nexample}

\begin{nexample}[]
    The labelling of the diamond $\{a < (b\,|\,c) < d\}$ in $\Set$ depicted below is a local functor:
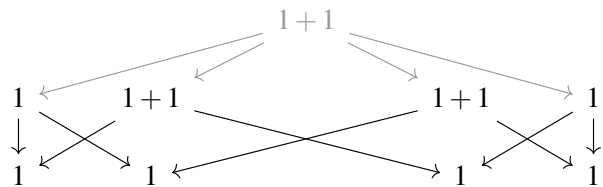
\[\begin{tikzcd}[ampersand replacement=\&, row sep=13pt]
        \&\& \textcolor{black!40}{1+1} \\
        1 \& {1+1} \&\& {1+1} \& 1 \\
        1 \& 1 \&\& 1 \& 1
        \arrow[from=2-4, to=3-5]
        \arrow[from=2-1, to=3-2]
        \arrow[from=2-1, to=3-1]
        \arrow[from=2-2, to=3-1]
        \arrow[from=2-2, to=3-4]
        \arrow[from=2-5, to=3-4]
        \arrow[from=2-5, to=3-5]
        \arrow[color=black!40, from=1-3, to=2-4]
        \arrow[color=black!40, from=1-3, to=2-1]
        \arrow[color=black!40, from=1-3, to=2-2]
        \arrow[color=black!40, from=1-3, to=2-5]
        \arrow[from=2-4, to=3-2]
\end{tikzcd}\]
\end{nexample}

Though this technical condition seems cumbersome to check explicitly, a remarkable fact about our formalism is that for a rather large class of poset shapes we almost never actually need to do this.
We will prove in the remainder of the paper that so long as the diagrams are constructed by taking finite limits of other local diagrams of the right shape, we will remain within this fragment of our framework.
For our interest in using this combinatorial framework as the basis of a future proof assistant, this shows we can maintain strong invariants on our data structures, so that any actual implementation of our procedures could drastically reduce the data it needs to explicitly keep track of.

Another, even more surprising property of our formalism is that this class of well-behaved poset shapes can be morally taken to be the whole of $\FinPos$, albeit viewed through a looking glass.
This is because we may take our well-behaved posets to be finite distributive lattices,
by which we mean posets $P$ admitting finite meets and joins
and that satisfying the condition that for all $a,b,c \in P$ we have
$a \land (b \lor c) = (a \lor b) \land (a \lor c)$ and $a \lor (b \land c) = (a \land b) \lor (a \land c)$.
Such posets assemble into a non-full subcategory $\FinDLat$ of $\FinPos$ by taking as maps  monotone functions $f: P \to Q$ preserving finite meets and joins.
Note that, in particular, our lattices are always bounded, and moreover if $f$ is a lattice homomorphism then we must have $f(\bot) = \bot$ and $f(\top) = \top$.
Although in our formal development some individual results would hold with weaker regularity conditions,
the category $\FinDLat$ enjoys the property of being equivalent to $\FinPos^\Op$ by the Birkhoff Representation Theorem~\cite[page 262]{wraithusing1993}.
The sum of these wonderful properties suggests us the following definition:
\begin{ndef}[Labelled Posetal Diagram]
    The category $\constr{P}(\cat{C})$ of \emph{posetal diagrams labelled in} $\cat{C}$ is defined to be the subcategory of labelled intervals $\constr{L}(\cat{C})$ with objects pairs $(P,\func{X})$ with $P$ being a distributive lattice and $\func{X}$ a local functor, and morphisms pairs $(f,\alpha)$, with $f$ a lattice homomorphism.
\end{ndef}

\subsection{Intervals in Lattices}
\label{subsec:intinlat}

Having formally introduced our key notion of posetal diagrams in Section \ref{subsec:posdiaglocal}, we will now embark a fine-grained analysis of the relation between logical properties of a poset $P$ and the locality property of the labelled interval with shape $P$.
Our key result for this Section will be an explicit characterisation of atomic cospans in $[P]$ for distributive lattices.
This will allow us to extract useful consequences about preservation of locality under proposition by interval maps associated with lattice homomorphisms.
To do this, however, we first require some intermediate lemmas.
\begin{nlemma}[]
    \label{lemma:intervalmeetsemilat}
    Let $P$ be a finite poset with binary meets and joins.
    Then, for any two intervals $[a,a']$ and $[b,b']$ in $[P]$, the meet $[a,a'] \land [b,b']$ exists and is given by $[a \land b, a' \lor b']$.
\end{nlemma}
\begin{proof}
    We have $a \land b \le a$ and $a' \le a' \lor b'$, and similarly for $[b,b']$, hence $[a \land b, a' \lor b']$ is an interval and $[a \land b, a' \lor b'] \supseteq [a,a'],[b,b']$.
    Moreover, for any other interval $[c,c']$ in $P$ with $[c,c'] \supseteq [a,a'],[b,b']$, we must have $c \le a$ and $c \le b$, and thus $c \le a \land b$.
    The dual calculation shows $[c,c'] \supseteq [a \land b, a' \lor b']$.
\end{proof}

\begin{nlemma}[]
\label{lemma:intervaljoinexist}
    Let $P$ be a finite poset with binary meets and joins.
    Then, for any two intervals $[a,a']$ and $[b,b']$ in $[P]$, the join $[a,a'] \lor [b,b']$ exists iff $a \lor b \le  a' \land b'$, in which case it is given by $[a \lor b, a' \land b']$.
\end{nlemma}
\begin{proof}
	Assume $[a,a'] \lor [b,b']$ exists and equal to some interval $[c,c']$.
	Then in particular we have $a \le c$ and $b \le c$, and thus $a \lor b \le c$, and dually $c' \le a' \land b'$.
	Since $c \le c'$, we must have $a \lor b \le a' \land b'$.

	We establish the other direction by verifying the universal property of the interval $[a \lor b,a' \land b']$, whenever it exists.
	We have $[a,a'] \supseteq [a \lor b,a' \land b']$, and similarly for $[b,b']$,
	and moreover for every interval $[c,c']$ satisfying the containments $[a,a'] \supseteq [c,c'] \subseteq [b,b']$, we must have $[a \lor b,a' \land b'] \supseteq [c,c']$ by the above verifications.
\end{proof}

\begin{nprop}[]
\label{prop:latatomiciffjoin}
    Let $P$ be a finite lattice.
    Then a cospan $[a,a'] \supseteq [b,b'] \subseteq [c,c']$ in $[P]$ is atomic
	iff $a \lor c \le a' \land c'$ and $[b,b'] = [a \lor c, a' \land c']$.
\end{nprop}
\begin{proof}
    In the forward direction, we have that $a \le b$ and $c \le b$, hence $a \lor c \le b$ and dually $b' \le a' \land c'$.
    In particular, $[a \lor c, a' \land c'] \supseteq [b,b']$, and thus we have $[a \lor c, a' \land c'] = [b,b']$ by atomicity and anti-symmetry.
    The backward direction is given by the universal property of $a \lor c$ and $a' \land c'$.
\end{proof}

This yields an immediate but essential consequence for our forthcoming study of limits of posetal diagrams in Section \ref{subsec:limposdiagloc}:
\begin{ncor}[]
\label{cor:lathompreservesatomiccospans}
	If $f: P \to Q$ is a lattice homomorphism, $[f]$ preserves atomic cospans and their pullbacks.
\end{ncor}
\begin{proof}
    By Proposition \ref{prop:latatomiciffjoin}, a cospan $[a,a'] \supseteq [b,b'] \subseteq [c,c']$ is atomic iff $[b,b'] = [a \lor c, a' \land c']$.
    But lattice homomorphisms preserve binary meets and joins, so $[f][b,b'] = [f(a) \lor f(c), f(a') \land f(c')]$,
    hence the cospan $[f][a,a'] \supseteq [f][b,b'] \subseteq [f][c,c']$ is atomic.
    Moreover, $[P]$ is a thin category, so the pullback of our atomic cospan coincides with the product of $[a,a']$ and $[c,c']$.
    By Lemma \ref{lemma:intervalmeetsemilat}, this is given by $[a \land c, a' \lor c']$, and thus it is preserved by $f$.
\end{proof}

\section{Limits of Posetal Diagrams}
\label{sec:limofdiags}

\subsection{Limit Procedure}
\label{subsec:limalg}

Having given the construction of the category $\constr{L}(\cat{C})$ in our preceding Section,
we will now describe a procedure to compute limits in $\constr{L}(\cat{C})$ from limits in $\cat{C}$ and $\FinPos$.
By taking the category of labels $\cat{C}$ to be a suitable $n$-fold iterate of the labelled interval construction $\constr{L}^n(\cat{D})$, this procedure provides a recursive strategy for computing limits entirely in terms of those in the base category $\cat{D}$.

Before we describe the limit procedure, since we are interested in stating our results for categories which may be fail to admit many limits, we need to fix some terminology.
Recall that for functors $\func{F}: \cat{J} \to \cat{C}$ and $\func{G}: \cat{C} \to \cat{D}$, we say that $\func{G}$ preserves $\func{F}$-limits if whenever $(L,\eta: \Delta_L \to \func{F})$ is a limit for $\func{F}$, $(\func{G}L,\func{G}\eta)$ is a limit for $\func{G} \circ \func{F}$.
We also say that $\func{G}$ reflects $\func{F}$-limits if whenever we have a cone $(L,\eta)$  over $\func{F}$ such that $(\func{G}L,\func{G}\eta)$ is a limit for $\func{G} \circ \func{F}$, then $(L,\eta)$ is a limit for $\func{F}$ \cite[3.3.1]{riehlcategory2016}.
\begin{ndef}[Pointwise Limit]
    Let $(L,\eta)$ be a limit for a diagram $\func{F}: \cat{J} \to \Func(\cat{C},\cat{D})$.
    We say $(L,\eta)$ is \emph{pointwise} if for each $c \in \cat{C}$, it is preserved by the evaluation at $c$ functor $\ev_c: \Func(\cat{C},\cat{D}) \to \cat{D}$.
\end{ndef}

If $\cat{D}$ has all $\cat{J}$-limits, every $\cat{J}$-limit in $\Func(\cat{C},\cat{D})$ is pointwise, but this need not be the case otherwise, and our analysis will necessitate the distinction.

\begin{nconstr}[Limit Procedure]
\label{constr:limit}
    Given a category of labels $\cat{C}$ and a finite diagram $\func{F}: \cat{J} \to \constr{L}(\cat{C})$,
    we compute the limit of $\func{F}$ or fail according to the following procedure:
    \begin{enumerate}
        \item We take the limit $(L,\rho)$ of the diagram $\func{U} \circ \func{F}: \cat{J} \to \constr{L}(\cat{C}) \to \FinPos$,
		\item We use $\rho$ to produce from $\func{F}$ a diagram $\func{G}: \cat{J} \to \Func([L],\cat{C})$,
		\item For each $j\in \cat{J}$, we take $\func{G}j := (\func{F}j)^1 \circ [\rho_j]$,
		\item For each $h \in \cat{J}(j,j')$, we take $\func{G}h: \func{G}j \to \func{G}j'$ to have components
		\begin{align*}
			(\func{G}h)_{[a,b]} := (\func{F}h)^1_{[\rho_j][a,b]}: (\func{F}j)^1 [\rho_j][a,b] &\longrightarrow (\func{F}j')^1 [\rho_{j'}][a,b]
		,\end{align*}
		\item We define $\varepsilon: (\Delta_L,\func{G}) \to \func{F}$ to have components $\varepsilon_{j}:= (\rho_j,\id_{(\func{G}j)^1})$,
		\item For each interval $[a,b]$ in $L$, we take the limit $(\func{L}_{[a,b]},\eta_{[a,b]})$ of $\ev_{[a,b]} \circ \func{G}: \cat{J} \to \cat{C}$, if it exists, and fail if not,
		\item For each pair $[a,b] \supseteq [c,d]$ of intervals in $L$, we use naturality of $\ev_-$ to get cones $(\func{L}_{[a,b]}, \ev_{[a,b] \supseteq [c,d]} \circ \eta_{[a,b]})$ over $\ev_{[c,d]} \circ \func{G}$,
		\item We get cone maps $\func{L}_{[a,b] \supseteq [c,d]}: \func{L}_{[a,b]} \to \func{L}_{[c,d]}$ via the u.p. of $(\func{L}_{[c,d]},\eta_{[c,d]})$,
		\item We assemble the above into a limit cone $(\func{L},\eta)$ with $\func{L}: [a,b] \mapsto \func{L}_{[a,b]}$, $([a,b] \supseteq [c,d]) \mapsto \func{L}_{[a,b] \supseteq [c,d]}$ and $(\eta_j)_{[a,b]}:= (\eta_{[a,b]})_j$.
		\item We return the cone $((L,\func{L}),\eta \circ \varepsilon)$ as a limit for $\func{F}$.
	\end{enumerate}
\end{nconstr}

\begin{figure}
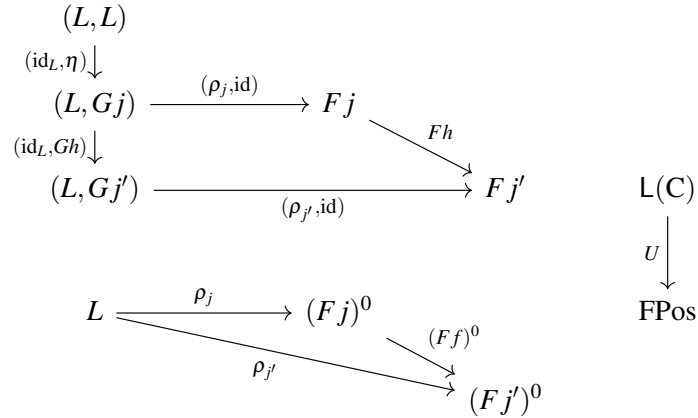

\figlimalgorithm
\vspace{-10pt}
\caption{A diagrammatic presentation of Construction \ref{constr:limit}.}
\label{fig:limalg}
\end{figure}

The end-to-end procedure is depicted in Figure \ref{fig:limalg}.

\begin{nlemma}[]
    \label{lemma:contractwelldefn}
    For every finite diagram $\func{F}: \cat{J} \to \constr{L}(\cat{C})$,
    Construction \ref{constr:limit} is well-defined.
\end{nlemma}
\begin{proof}
	The limit at Step (1) exists because $\FinPos$ is finitely complete \cite[12.6.1]{adamekabstract1990}.
	Since $(L,\rho)$ is a cone, for every $h \in \cat{J}(j,j')$, we have $(\func{F}h)^0 \circ \rho_j = \rho_{j'}$, hence $\func{G}h$ is well-defined.
	Moreover, since $\func{F}$ and $[-]$ are functors, $\func{G}$ respects identities,
	and since for $h \in \cat{J}(j,j')$ and $h' \in \cat{J}(j',j'')$, we have
	$(\func{F}(h' \circ h)^1)_{[\rho_j][a,b]} = (\func{F}(h') \circ [\rho_{j'}]) \circ \func{F}h)^1)_{[\rho_j][a,b]}$,
	$\func{G}$ respects composites and thus is a functor.
	The naturality condition for $\varepsilon$ in Step (5) follows by unwinding definitions.
	Functoriality of $\func{L}$ in Step (9) follows by uniqueness of the cone maps,
	and finally,
	naturality of $\eta$ holds along $\cat{J}$ due to $(\eta_-)_{[a,b]} := \eta_{[a,b]}$ being a cone,
	and along $[P]$ due to $\func{L}(-\supseteq-)$ being a map of cones.
\end{proof}

\begin{nprop}[]
	\label{prop:basechangecont}
	Let $f: P \to Q$ be a monotone function and $\cat{C}$ a category.
	Then $- \circ [f]: \Func([Q],\cat{C}) \to \Func([P],\cat{C})$ preserves all pointwise limits which exist in $\Func([Q],\cat{C})$.
\end{nprop}
\begin{proof}
	Let $(\func{L},\eta)$ be a pointwise limit for a diagram $\func{G}: \cat{J} \to \func{L}(Q)$,
	and let $(\func{K},\chi)$ be a cone over $\func{G} \circ [f]$.
	We define a natural transformation $\gamma: \func{K} \to \func{L} \circ [f]$ by taking for each interval $[a,b]$ in $P$,
	the component $\gamma_{[a,b]}$ to be the unique map of cones $(\func{K}[a,b],\ev_{[a,b]}\chi) \to (\func{L}f[a,b],\ev_{[f][a,b]}\eta)$, which is given by the universal property of the pointwise limit.
	Naturality and uniqueness of $\gamma$ then follow both by uniqueness of the components.
\end{proof}

\begin{nthm}[]
	\label{thm:contractsound}
	Let $\cat{C}$ be a category.
	If $\func{F}: \cat{J} \to \constr{L}(\cat{C})$ is a finite diagram,
	and Construction \ref{constr:limit} succeeds for $\func{F}$, its output $((L,\func{L}),\varepsilon \circ \eta)$ is a limit for $\func{F}$.
\end{nthm}
\begin{proof}
	By Lemma \ref{lemma:contractwelldefn}, Construction \ref{constr:limit} is well-defined.
	Let $((K,\func{K}),\gamma)$ be a cone over $\func{F}$.
	Since $L$ is a limit for $\func{U} \circ \func{F}$, we have a unique map $k: K \to L$.
	Since $(\func{L},\eta)$ is a pointwise limit, by Proposition \ref{prop:basechangecont}, $- \circ [k]$ preserves it, so we have a unique map of cones $\chi: \func{K} \to \func{L}$, so we take $(k,\chi)$ as our map.
\end{proof}

\begin{ncor}[]
\label{cor:jcompletelabelint}
    Let $\cat{J}$ be a finite category.
	If a category $\cat{C}$ has $\cat{J}$-limits, then so does $\constr{L}(\cat{C})$.
\end{ncor}
\begin{proof}
	This is a known result about fibred categories, see e.g. \cite[Thm.1]{tarleckifundamental1991}, which we can extract as a consequence of Theorem \ref{thm:contractsound}.
	Since $\cat{C}$ has $\cat{J}$-limits, so will each functor category $\Func([P],\cat{C})$, as they inherit the pointwise limits from $\cat{C}$.
	Hence Construction \ref{constr:limit} will succeed for all diagrams $\func{F}:\cat{J} \to \constr{L}(\cat{C})$, and thus by Theorem \ref{thm:contractsound} every such diagram has a limit.
\end{proof}

For a converse statement, we can prove that if $\cat{C}$ has an initial object, then our Construction \ref{constr:limit} is searching for a limit in the correct fibre:
\begin{nprop}[]
\label{prop:fibpresliminit}
	Let $\cat{C}$ be a category with an initial object $0 \in \cat{C}$.
	The fibration $\func{U}: \constr{L}(\cat{C}) \to \FinPos$ preserves all existing limits.
\end{nprop}
\begin{proof}
	The functor $\func{U}$ has a left-adjoint $\func{F}: \FinPos \to \constr{L}(\cat{C})$,
	which acts as $P \mapsto (P,\Delta_0: [P] \to \cat{C})$ and $f \mapsto (f,!)$,
	where $\Delta_0$ is the constant functor on $0 \in \cat{C}$.
\end{proof}

The following lemma allows us to safely invoke completion arguments:
\begin{nlemma}[]
\label{lemma:relabelpostcomppresrefl}
	Let $\func{F}: \cat{C} \to \cat{D}$ be a functor.
	If $\func{F}$ preserves, resp. reflects, all $\cat{J}$-limits which exist in $\cat{C}$, resp. $\cat{D}$,
	then $\constr{L}(\func{F}): \constr{L}(\cat{C}) \to \constr{L}(\cat{D})$ preserves, resp. reflects, all $\cat{J}$-limits produced by Construction \ref{constr:limit}.
\end{nlemma}
\begin{proof}
	For preservation, let $((L,\func{L}),\eta)$ be a limit for a diagram $\func{G}: \cat{J} \to \constr{L}(\cat{C})$ obtained via Construction \ref{constr:limit}.
	Then $\constr{L}(\func{F})$ sends this limit to $((L,\func{F} \circ \func{L}),\varepsilon)$,
	where $\varepsilon := (\eta_j^0,\func{F} \eta_j^1)$.
	Since $\func{F}$ preserves $\cat{J}$-limits in $\cat{C}$, this cone matches the output of Construction \ref{constr:limit} for $\constr{L}(\func{F}) \circ \func{G}$, and thus by Theorem \ref{thm:contractsound} is limiting.

	For reflection, let $((L,\func{L}),\eta)$ be a cone over a diagram $\func{G}: \cat{J} \to \constr{L}(\cat{C})$ which is mapped under $\constr{L}(\func{F})$ to the limit cone $((K,\func{K}),\varepsilon)$
	obtained via Construction \ref{constr:limit} on $\constr{L}(\func{F}) \circ \func{G}$.
	Since $\constr{L}(\func{F})$ acts only on labels, we must have $L = K$.
	Moreover, since the limit $(L,\func{K})$ is pointwise and $\func{F}$ reflects $\cat{J}$-limits in $\cat{D}$, $((L,\func{L}),\eta)$ satisfies the specification of Construction \ref{constr:limit}, and thus is a limit for $\func{G}$.
\end{proof}

\subsection{Limits of Posetal Diagrams}
\label{subsec:limposdiagloc}

Having presented a procedure for computing limits in $\constr{L}(\cat{C})$, we will conclude our technical exposition with a study of the corresponding limit procedure for the subcategory $\constr{P}(\cat{C})$, and extract some consequence for local diagrams.
A lot of the heavy lifting of our results in this section hinges on the following lemma, which allows a translation of Construction \ref{constr:limit} from labelled intervals to posetal diagrams.
\begin{nlemma}[]
\label{lemma:dlatposcont}
	The subcategory inclusion $\FinDLat \to \FinPos$ preserves finite limits.
\end{nlemma}
\begin{proof}
	Let $(L,\eta)$ be a limit cone for a finite diagram $\func{F}: \cat{J} \to \FinDLat$.
	By the Birkhoff Representation Theorem \cite[page 262]{wraithusing1993},
	the functors $\FinPos(-,2): \FinPos^\Op \to \FinDLat$ and $\FinDLat(-,2): \FinDLat^\Op \to \FinPos$ form an adjoint equivalence,
	where 2 denotes the two-element distributive lattice $\{\bot \to \top\}$ and
	the hom-sets are equipped with their respective pointwise orders.
	Since $\FinDLat(-,2)$ is a right adjoint, it preserves the limit $(L,\eta)$,
	sending it to the colimit of $\FinDLat(\func{F}-,2): \cat{J}^\Op \to \FinPos$.

	It suffices now to show that the composite of the dualising functor $\FinPos(-,2)$ and subcategory inclusion $\FinDLat \to \FinPos$ preserves finite colimits.
	But this is just given by the internal contravariant hom $\FinPos(-,2): \FinPos^\Op \to \FinPos$ for the Cartesian closed category $\FinPos$ \cite[27.3.1]{adamekabstract1990}.
	By the enriched variant of the familiar continuity result \cite[3.29]{kellybasic1982},
	this sends the colimit $(\FinDLat(L,2),\FinDLat(\eta,2))$
	to the limit $(\FinPos(\FinDLat(L,2),2), \FinPos(\FinDLat(\eta,2),2))$
	of the composite diagram $\cat{J} \to \FinPos$,
	which, by the duality, is isomorphic to the image of the initial cone under the inclusion $\FinDLat \to \FinPos$.
\end{proof}

\begin{nprop}[]
	\label{prop:posdiagsubcateqclosed}
    If $\cat{C}$ has all equalisers, then $\constr{P}(\cat{C})$ is closed under equalisers taken in $\constr{L}(\cat{C})$.
\end{nprop}
\begin{proof}
	Let us identify $\constr{P}(\cat{C})$ with its image in $\constr{L}(\cat{C})$,
	and consider the parallel pair $(f,\alpha),(g,\beta): (P,\func{X}) \to (Q,\func{Y})$ of posetal maps.
	Let $(E,e:E \subseteq P)$ be the equaliser of $f$ and $g$ in $\FinPos$,
	where $E$ is identified with its image in $P$ under $e$, and $(\func{W},\gamma)$ be the equaliser of $\alpha_{[e]}$ and $\beta_{[e]}$ in $\Func([E],\cat{C})$.
	We wish to show $\func{W}$ is local,
	so let $[a,a'] \supseteq [b,b'] \subseteq [c,c']$ be an atomic cospan of intervals in $E$ with pullback $[d,d']$.
	By Lemma \ref{lemma:dlatposcont}, $e$ is a lattice homomorphism,
	so by Corollary \ref{cor:lathompreservesatomiccospans},
	the image of the cospan under $e$ is an atomic cospan in $P$ with pullback $[d,d']$,
	and similarly for $[f][d,d']$ under $f \circ e = g \circ e$.
	Since $\func{X}$ and $\func{Y}$ are local, we are left to prove that the square $\func{W}([d,d'] \supseteq [a,a'],[c,c'] \supseteq [b,b'])$ of pointwise equalisers of two pullback squares is again a pullback.
	Our result hence follows either by general preservation of limits by limits, or by a diagram chase on the following:
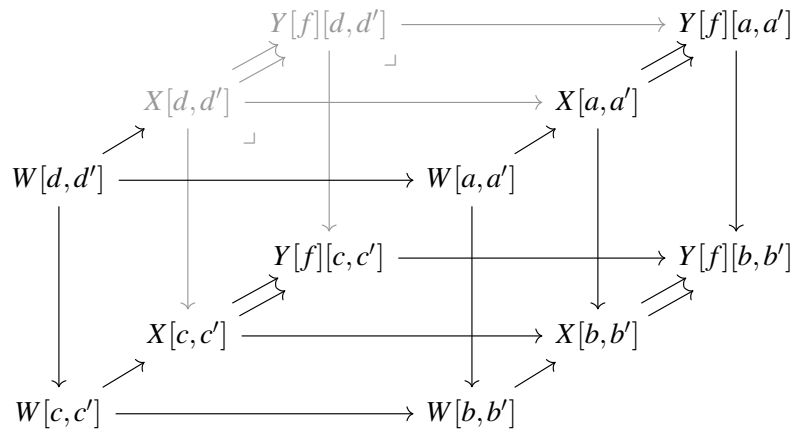
\figequaliserofpullbacksquares
\end{proof}

\begin{nprop}[]
\label{prop:posdiagsubcatprodclosed}
    Let $\cat{C}$ be a category with finite products.
    The subcategory $\constr{P}(\cat{C})$ is closed under finite products taken in $\constr{L}(\cat{C})$.
\end{nprop}
\begin{proof}
    Let $(P,\func{X}),(Q,\func{Y})$ be posetal diagrams,
    and denote their product in $\constr{L}(\cat{C})$ by $(P\times Q, \func{W})$.
    By Lemma \ref{lemma:prodintintprod}, we have an isomorphism $[P \times Q] \cong [P] \times [Q]$,
    under which $\func{W}$ acts by $([a,a'],[b,b']) \mapsto \func{X}[a,a'] \times \func{Y}[b,b']$.
    By Lemma \ref{lemma:dlatposcont}, $P \times Q$ is a lattice, and thus by Corollary \ref{cor:lathompreservesatomiccospans} the projections $[P\times Q] \to [P],[Q]$ preserve atomic cospans and their pullbacks.
    Since $\func{X},\func{Y}$ are local, so is $\func{W}$.
    Furthermore, the terminal object $(1,\Delta_1)$ is always local, so the result holds.
\end{proof}

\begin{ncor}
	\label{cor:mainresult}
	If $\cat{C}$ has all finite limits, then $\constr{P}(\cat{C})$ has all finite limits.
\end{ncor}
\begin{proof}
	By Proposition \ref{prop:posdiagsubcateqclosed} and Proposition \ref{prop:posdiagsubcatprodclosed}, $\constr{P}(\cat{C})$ is closed under arbitrary finite limits in $\cat{L}(\cat{C})$.
	But by Corollary \ref{cor:jcompletelabelint}, $\constr{L}(\cat{C})$ is finitely complete, hence so is $\constr{P}(\cat{C})$.
\end{proof}

\bibliographystyle{eptcs}
\bibliography{bibliography}
\end{document}